\documentclass[a4paper,11pt]{amsart}

\usepackage[utf8]{inputenc}

\usepackage[english]{babel}
\usepackage{amsmath,amsfonts,amsthm,xcolor,amssymb}
\usepackage{tikz}
\usetikzlibrary{positioning}
\usepackage[footnotesize]{caption}
\usepackage{geometry}
\usepackage[mathscr]{euscript}
\usepackage[colorlinks=true, citecolor=blue]{hyperref}
\usepackage[shortlabels]{enumitem}
\usepackage[capitalise]{cleveref}
\usepackage{mleftright}
\usepackage{dsfont}
\usepackage{multirow,tabularx}
\newcolumntype{Y}{>{\centering\arraybackslash}X}
\numberwithin{equation}{section}

\theoremstyle{plain}
\newtheorem{theorem}{Theorem}[section]
\newtheorem{lemma}[theorem]{Lemma}
\newtheorem{corollary}[theorem]{Corollary}
\newtheorem{proposition}[theorem]{Proposition}
\theoremstyle{definition}
\newtheorem{remark}[theorem]{Remark}

\newtheorem{example}[theorem]{Example}
\newtheorem{definition}[theorem]{Definition}

\newcommand{\R}{\mathbb{R}}
\newcommand{\C}{\mathbb{C}}
\newcommand{\N}{\mathbb{N}}
\newcommand{\F}{\mathbb{F}}
\renewcommand{\P}{\mathbb{R}\mathrm{P}}
\DeclareMathOperator{\Prob}{\mathbf P}

\title{Typical ranks of random order-three tensors}

\author[P.~Breiding]{Paul Breiding}
\address[Breiding]{University Osnabr\"uck,
FB 06 Mathematik/Informatik/Physik
Albrechtstr.~28a,
49076 Osnabr\"uck, Germany}
\email{pbreiding@uni-osnabrueck.de}

\author[S.~Eggleston]{Sarah Eggleston}
\address[Eggleston]{University Osnabr\"uck,
FB 06 Mathematik/Informatik/Physik
Albrechtstr.~28a,
49076 Osnabr\"uck, Germany}
\email{seggleston@uni-osnabrueck.de}

\author[A.~Rosana]{Andrea Rosana}
\address[Rosana]{Scuola Internazionale Superiore di Studi Avanzati (SISSA),
via Bonomea, 265,
34136 Trieste, Italy}
\email{arosana@sissa.it}

\begin{document}

\begin{abstract}
In this paper we study typical ranks of real $m\times n \times \ell$ tensors. In the case $ (m-1)(n-1)+1 \leq \ell \leq mn$ the typical ranks are contained in $\{\ell, \ell +1\}$, and $\ell$ is always a typical rank. We provide a geometric proof of this fact. We express the probabilities of these ranks in terms of the probabilities of the numbers of intersection points of a random linear space with the Segre variety.
In addition, we give some heuristics to understand how the probabilities of these ranks behave, based on  asymptotic results on the average number of real points in a random linear slice of a Segre variety with a subspace of complementary dimension.

The typical ranks of real $3\times 3\times 5$ tensors are $5$ and $6$. We link the rank probabilities of a $3\times 3 \times 5$ tensor with i.i.d.\ Gaussian entries to the probability of a random cubic surface in $\P^3$ having real lines. As a consequence, we get a bound on the expected number of real lines on such a surface.
\end{abstract}

\maketitle

\vspace{-0.6cm}
\section{Introduction}

The determination of tensor rank has broad applications in many fields, including mathematics, psychometrics, and data science \cite{Kolda2009}.
This paper focuses on \emph{typical ranks}.
A typical rank of a given format is any number~$r$ such that the set of tensors of rank $r$ is a full-dimensional semialgebraic set \cite[\S 7]{Friedland2012}.  While over the complex numbers (in fact, over any algebraically closed field) there is only one typical rank,  much less is known about typical ranks of real tensors.

Here, we focus on the family of real order-three tensors of the format $m\times n\times \ell$ where $2\leq m\leq n\leq \ell$. We take the perspective of probability and define typical ranks of tensors as follows.
\begin{definition}\label{def_typical}
We call $T=(t_{i,j,k})\in\mathbb R^{m\times n\times \ell}$ a \emph{Gaussian tensor} if its entries $t_{i,j,k}$ are i.i.d.\ standard Gaussian. An integer $r>0$ is called a \emph{typical rank} for the format $m\times n\times \ell$ if for a Gaussian tensor $T$ we have $\Prob\{\mathrm{rank}(T)=r\}>0$. 
\end{definition}
Our first main theorem is the following. The proof comes in \cref{sec:friedland_prob}.
\begin{theorem}\label{main0_intro}
Suppose that $(m-1)(n-1)+1 \leq \ell \leq mn$. Then, the typical ranks of real $m\times n\times \ell$ tensors are contained in $\{\ell, \ell+1\}$. Moreover, $\ell$ is always a typical rank.
\end{theorem}
In \cite[Theorem 8.1]{SMS17} the authors give an algebraic proof of \cref{main0_intro} assuming $3\leq m$. Our proof, on the other hand, is geometric and is based on the \emph{General Position Theorem} for general linear sections of algebraic varieties (see \cref{prop_linear_sections}). 

The lower bound of $\ell=(m-1)(n-1)+1$ in \cref{main1_intro} is the codimension of the variety of rank-one matrices in $\mathbb R^{m\times n}$, called the \emph{Segre variety}. This is a real projective variety that we denote by $\Sigma_{m,n}\subset \P^{mn-1} $.
Our next theorem links rank probabilities to the probabilities of having a certain number of intersections points of $\Sigma_{m,n}$ with a \emph{uniform subspace $L\subset \P^{mn-1}$}. This random subspace is the projectivization of a uniform random variable in the Grassmannian $G(\ell,mn)$ of $\ell$-dimensional linear spaces in $\mathbb R^{m\times n}$ (see \cref{sec:random_subspaces} for more details).
We give a proof of the theorem in \cref{sec:friedland_prob}.

\begin{theorem}\label{main1_intro}
Let $T \in \R^{m\times n \times \ell}$.
If $\ell = (m-1)(n-1)+1$, we have 
$$\underset{T \in\mathbb R^{m\times n\times \ell} \atop \mathrm{ Gaussian}}{\Prob}\{\mathrm{rank}(T) = \ell\} =
\underset{L\in G(\ell, mn) \atop \mathrm{ uniform}}{\Prob} \{\#(L \cap \Sigma_{m,n})\geq \ell \}.$$
For $(m-1)(n-1)+1 < \ell \leq mn$ we have 
$$\underset{T \in\mathbb R^{m\times n\times \ell} \atop \mathrm{ Gaussian}}{\Prob}\{\mathrm{rank}(T) = \ell\} =
\underset{L\in G(\ell, mn) \atop \mathrm{ uniform}}{\Prob} \{L \cap \Sigma_{m,n} \neq \emptyset\}.$$
\end{theorem}

In particular, the second part of this theorem proves that, when $(m-1)(n-1)+1 < \ell$ and the degree of the \emph{complex} Segre variety is odd, then we have only one typical rank~$\ell$. Recall from \cite[Example 18.15]{harris} that the degree of the complex Segre variety is~$\binom{m+n-2}{m-1}$. By Lucas' theorem \cite{lucas}, this number is odd if and only if the binary expansions $m+n-2=\sum_{i\geq 0} a_i2^i$ and $m-1=\sum_{i\geq 0} b_i2^i$ satisfy $(\forall i: b_i=1\Rightarrow a_i=1)$. 

A currently open problem is which formats of real tensors have multiple typical ranks. The cases of $(m-1)(n-1)+1 < \ell \leq mn$ were addressed by \cite[Theorem 1.1]{SMS17}, where the authors showed that multiple typical ranks occur if and only if there exists a nonsingular bilinear map $\mathbb{R}^m \times \mathbb{R}^n \to \mathbb{R}^{mn-\ell}$. In their paper, such a bilinear map $\varphi$ is nonsingular if $\varphi(x,y) = 0$ implies that $x=0$ or $y=0$. This is equivalent to the existence of a linear space $L\in G(\ell, mn)$ such that $L \cap \Sigma_{m,n} =\emptyset$. If $L \cap \Sigma_{m,n} =\emptyset$, then all linear spaces in a neighborhood of $L$ also do not intersect $\Sigma_{m,n}$. Hence, \cite[Theorem 1.1]{SMS17} also follows from the second part of \cref{main1_intro}.

Our proof of \cref{main1_intro} relies on a result by Friedland, which we formulate geometrically in terms of linear sections of the Segre variety. Using this approach, in \cref{sec:tall_tensors} we moreover provide a proof of the fact that when $(m-1)n < \ell \leq mn$ there is always a unique typical rank. This was shown previously through algebraic means by~\cite{tenBerge2000,SSM16}.

Let us now turn our focus to  rank probabilities. In the case $(m,n,\ell)=(2,2,2)$, Bergqvist \cite{Bergqvist2013} showed that for $T\in\mathbb R^{2\times 2\times 2}$ Gaussian one has $\Prob\{\mathrm{rank}(T)=2\} = \tfrac{\pi}{4}$ and $\Prob\{\mathrm{rank}(T)=3\} = 1-\tfrac{\pi}{4}$. They also computed the rank probabilities for tensors of the format $(m,n,\ell)=(2,3,3)$. Later, Bergqvist and Forrester \cite{BF2011} computed the rank probabilities of $2\times n\times n$ Gaussian tensors for any $n$. Beyond that, we are not aware of any results on rank probabilities of Gaussian tensors.

Using a result by \cite{CAET2023} we can link the rank probabilities of Gaussian  $3\times 3\times 5$-tensors to the probabilities of a random cubic surface having real lines.
Recall that a smooth real cubic surface contains $27$ complex lines, and that either $3,7,15$ or $27$ of them are real lines. We have the following result.
\begin{theorem}\label{thm:rank-reallines_intro}
Let $T\in\mathbb R^{3\times 3\times 5}$ be a Gaussian tensor, and let $M_0,\dots,M_3 \in \R^{3\times 3}$ be four independent Gaussian matrices. Define a random cubic surface
$$S := \{ [z_0:\dots:z_3] \in \P^3 \mid \det(z_0 M_0 + \dots + z_3 M_3) = 0 \}.$$
Then the following probabilities are equal
$$\Prob\{\mathrm{rank}(T) = 5\}= \Prob\{\text{the 27 lines on }S\text{ are all real}\}.$$
\end{theorem}
The proof of this theorem comes in \cref{sec:3x3x5}. The proof is based on \cref{cor_cubic_lines}, in which we relate the numbers of real lines to random linear sections of the Segre variety. As a consequence, we can prove that the expected number of lines of the cubic in \cref{thm:rank-reallines_intro} is at least $11$ and at most $15$; see \cref{cor_E}.

We close this paper in \cref{sec:heur} with a study of the asymptotics of Gaussian tensors of the format $m\times n\times (m-1)(n-1)+1$. We argue that for $m\geq 4$, when $n \to \infty$ one should expect the probability of the generic rank (the smallest typical rank) to go to 1. 

\subsection*{Acknowledgements}
P.~Breiding is funded by the Deutsche Forschungsgemeinschaft (DFG) – Projektnummer 445466444.

\bigskip
\section{Preliminaries on Typical Ranks}

In the following, we will always take $m,n,\ell \in \N$ such that $2 \leq m \leq n \leq \ell$.

Recall that the rank of an $m\times n\times \ell$ tensor $T$ over a field $\F$ is the minimum number of rank-one tensors that sum to $T$; i.e., the minimum number $r$ such that there exist $\{a_i\}_{i=1}^r \subset \F^m, \{b_i\}_{i=1}^r \subset \F^n$ and $\{c_i\}_{i=1}^r \subset \F^\ell$ with
$$T = \sum_{i=1}^r a_i \otimes b_i \otimes c_i .$$
For $\F=\C$ and fixed $m,n,\ell$, general tensors all have the same rank -- the generic rank of that format. Over $\R$, however, the story is different. Recall from \cref{def_typical} that we refer to all $r\in\N$ that arise with positive probability as ranks of a Gaussian tensor of a given format as typical ranks of that format. 

For many choices of tuples $(m,n,\ell)$, there is only one typical rank (equal to the generic rank of complex tensors of the given format). For instance, $3\times 3\times 3$ and $3\times 3\times 4$ tensors have only one typical rank, namely $5$ \cite{tBS2006,AOP09}. 
On the other hand, for some families of tensors there are multiple typical ranks. For example, tensors in~$\R^{2\times n \times \ell}$ with $2 \leq n \leq \ell$ have rank equal to $\min\{2n,\ell\}$ with probability one for all $n\neq \ell$, while for $n = \ell$, both $\ell$ and $\ell+1$ arise with positive probability \cite{tenBerge-Kiers99}. Known typical ranks of small order-three tensors are summarized in \cref{table:tB1}.

\begin{table}[ht]
\centering
\begin{tabularx}{\textwidth}{*{9}{Y}}
 \hline
  & \multicolumn{3}{c}{$m=2$} &&& \multicolumn{3}{c}{$m=3$} \\ [0.5ex] 
 \hline
  & $n = 2$ & $n = 3$ & $n = 4$ &&& $n = 3$ & $n = 4$ & $n = 5$ \\ 
 $\ell=2$ & \{2,3\} & 3 & 4 && $\ell=5$ & \{5,6\} & ? & ? \\
 $\ell=3$ & 3 & \{3,4\} & 4 && $\ell=6$ & 6 & ? & ? \\
 $\ell=4$ & 4 & 4 & \{4,5\} && $\ell=7$ & 7 & \{7,8\} & ? \\
 $\ell=5$ & 4 & 5 & 5 && $\ell=8$ & 8 & \{8,9\} & ? \\ 
 $\ell=6$ & 4 & 6 & 6 && $\ell=9$ & 9 & 9 & \{9,10\} \\ 
 $\ell=7$ & 4 & 6 & 7 && $\ell=10$ & 9 & 10 & 10 \\
 $\ell=8$ & 4 & 6 & 8 && $\ell=11$ & 9 & 11 & 11 \\ 
 $\ell=9$ & 4 & 6 & 8 && $\ell=12$ & 9 & 12 & 12 \\ [1ex] 
 \hline
\end{tabularx}
\caption{Typical ranks of real tensors of the format $m\times n\times \ell$, based on  the results \cite{tenBerge2004,SSM16,SMS21}. The symbol ``?'' indicates that the typical ranks of that format are not known.}
\label{table:tB1}
\end{table}

While purely algebraic approaches are common in determining typical ranks of tensors, tensors can also be viewed as geometric objects, and the rank is, correspondingly, a geometric property. The following theorem by \cite{Friedland2012} will be instrumental in this paper, as we follow a geometric approach to explore order-three tensor formats with multiple typical ranks. Angle brackets $\langle v_1,\dots,v_k\rangle$ denote the linear span of a set of vectors~$\{v_1,\dots,v_k\}$. 
\begin{theorem}[{\cite[Theorem 2.4]{Friedland2012}}]
    Fix $T \in \R^{m\times n\times \ell}$ and let $T_1, \dots, T_{\ell} \in \R^{m\times n}$ be the slices of $T$. Then the rank $r$ of $T$ is the minimal number of rank-one matrices $A_1,\dots,A_r \in \R^{m\times n}$ such that $\langle T_1,\dots,T_{\ell} \rangle \subseteq \langle A_1,\dots,A_r \rangle$.
\label{thm:Fri2.4}
\end{theorem}
For completeness, we give a short proof of this result.
\begin{proof}[Proof of Theorem \ref{thm:Fri2.4}]
We have $T=\sum_{i=1}^\ell T_i\otimes e_i
$, where $e_i\in\mathbb R^\ell$ is the $i$-th standard basis vector. On the one hand, if $T=\sum_{j=1}^r a_j\otimes b_j\otimes c_j$, write $c_j=\sum_{i=1}^\ell \lambda_{i,j}e_i$. Then, using multilinearity we have $T=\sum_{i=1}^\ell (\sum_{j=1}^r \lambda_{i,j} a_j\otimes b_j)\otimes e_i$. Therefore every $T_i$ is contained in $\langle a_1\otimes b_1,\dots,a_r\otimes b_r\rangle$. Conversely, if each $T_i$ is contained in $\langle a_1\otimes b_1,\dots,a_r\otimes b_r\rangle$, then there exists $\lambda_{i,j}$, such that $T=\sum_{i=1}^\ell (\sum_{j=1}^r \lambda_{i,j} a_j\otimes b_j)\otimes e_i = \sum_{j=1}^r a_j\otimes b_j \otimes (\sum_{i=1}^\ell \lambda_{i,j}e_i)$. That is, $T$ has rank at most $r$. 
\end{proof}

\bigskip
\section{Geometric perspective on typical ranks}\label{sec:generalization}
In \cref{thm:Fri2.4}, we seek to determine the minimal number of rank-one matrices that are needed to span the linear subspace $\langle T_1,\dots,T_{\ell} \rangle$. Here, we can consider the projectivization of that linear space because the rank of a matrix is invariant under non-zero scaling. Hence, we introduce the following notation.

The projective space over $\mathbb R^{m\times n}$ is $\P^{mn-1}$. Denote by
$$\Sigma_{m,n} := \{M\in\P^{mn-1} \mid \mathrm{rank}(M)=1\}.$$
This is a smooth algebraic variety in $\P^{mn-1}$, called \emph{Segre variety}, defined by the vanishing of $2\times 2$ minors of $M$. Its dimension is
$$\dim \Sigma_{m,n} = m+n-2.$$
The Segre variety is the image of the Segre embedding 
$\P^{m-1} \times \P^{n-1} \hookrightarrow \P^{mn-1}$, given by $(x,y)\mapsto x\otimes y.$
The complex Zariski closure of $\Sigma_{m,n}$ is the complex Segre variety $\Sigma_{m,n}^{\mathbb C} := \{M\in\mathbb C\mathrm P^{mn-1} \mid \mathrm{rank}(M)=1\}$. It is smooth and irreducible. By \cite[Example 18.15]{harris}, its degree is
\begin{equation}\label{degree_segre}
\mathrm{deg}(\Sigma_{m,n}^{\mathbb C}) = \binom{m+n-2}{m-1}.
\end{equation}
The latter means that a general linear subspace $L\subset \mathbb C\mathrm P^{mn-1}$ of codimension $m+n-2$ will intersect $\Sigma_{m,n}^{\mathbb C}$ in $\binom{m+n-2}{m-1}$ many complex points. Note that such an $L$ corresponds to an $((m-1)(n-1)+1)$-dimensional linear space in $\mathbb C^{m\times n}$. The word ``general'' here means that this holds for all $L$ outside some lower dimensional semialgebraic set.

\subsection{Random Linear Subspaces}\label{sec:random_subspaces}
For $1\leq d\leq mn$ we consider the \emph{Grassmannian}
$$G(d, mn) := \{L\subset \R^{m\times n}\mid L\text{ is a linear space of dimension } d\}.$$
The Grassmannian is a homogeneous space under the action of the orthogonal group $O(mn)$. There is a unique probability distribution on $G(d,mn)$ that is invariant under the action of $O(mn)$ on $G(d,mn)$ (see, e.g., \cite[Theorem 2.49 and Corollary 2.51]{folland}), which we call the \emph{uniform distribution}. 

The Euclidean inner product on $\mathbb R^{m\times n}$ is 
\begin{equation}\label{def_IP}
\langle A,B\rangle := \mathrm{Trace}(A^TB).
\end{equation}
For a subspace $L\in G(d,mn)$ we denote the \emph{orthogonal complement} by
\begin{equation}\label{def_OC}
L^\perp := \{A\in\mathbb R^{mn} \mid \langle A,B\rangle = 0 \text{ for all } B\in L\}.
\end{equation}

\begin{lemma}\label{lem_OC}
$L\in G(d,mn)$ is uniform if and only if $L^\perp\in G(mn-d,mn)$ is uniform.
\end{lemma}
\begin{proof}
$L\in G(d,mn)$ is $O(mn)$-invariant, if and only if $L^\perp\in G(mn-d,mn)$ is $O(mn)$-invariant. The result follows because there is a unique probability distribution on $G(d,mn)$ that is invariant under the action of $O(mn)$.
\end{proof}

Next, we study when uniform subspaces are spanned by rank-one matrices.
\begin{proposition}\label{prop_linear_sections}
Let $0\leq d\leq mn-1$ and $L\subset \P^{mn-1}$ correspond to a uniform linear subspace in $G(d+1,mn)$. 
With probability one the following hold.
\begin{enumerate}
    \item In the case $d < (m-1)(n-1)$: $L$ is not spanned by rank-one matrices 
    \item In the case $d = (m-1)(n-1)$: $\#(L\cap \Sigma_{m,n})\geq d+1$ implies that $L$ is spanned by rank-one matrices.
    \item In the case $d >(m-1)(n-1)$: $L\cap \Sigma_{m,n}\neq \emptyset$ implies that $L$ is spanned by rank-one matrices.
\end{enumerate}
In particular, for any $d$ we have that, with probability one, $\#(L\cap \Sigma_{m,n})\geq d+1$ implies that~$L$ is spanned by rank-one matrices.
\end{proposition}

\begin{proof}
With probability one, a uniform subspace in $G(d+1,mn)$ is a general subspace. Recall that the codimension of $\Sigma_{m,n}^{\mathbb C}$ in $\mathbb C\mathrm P^{mn-1}$ is $(m-1)(n-1)$.  For $L$ corresponding to a linear subspace in $G(d+1,mn)$, let $L^{\mathbb C}$ denote the complex linear space spanned by~$L$.

If $d < (m-1)(n-1)$, then with probability one we have  $L\cap \Sigma_{m,n} = \emptyset$, and so $L$ is not spanned by rank-one matrices.

If $d = (m-1)(n-1)$, with probability one $L^{\mathbb C} \cap \Sigma_{m,n}^{\mathbb C}$ consists of $ \mathrm{deg}( \Sigma_{m,n}^{\mathbb C}) = \binom{m+n-2}{m-1}$ complex points. The \emph{General Position Theorem} (see, e.g., 
\cite[p.~109]{arbarello2013geometry}) asserts that any subset of $d+1$ points in $L^{\mathbb C}\cap \Sigma_{m,n}^{\mathbb C}$ is linearly independent. Consequently, if there are at least $d+1$  points in $L\cap \Sigma_{m,n}$, there are at least $d+1$ linearly independent points in $L\cap \Sigma_{m,n}$. Their span must be $L$.

Finally, suppose that $d > (m-1)(n-1)$. In this case, $L^{\mathbb C}\cap \Sigma_{m,n}^{\mathbb C}$ is a positive dimensional projective variety. A general choice of $d+1$ points in $L^{\mathbb C}\cap \Sigma_{m,n}^{\mathbb C}$ spans~$L^{\mathbb C}$ by the \emph{General Position Theorem}. If $L\cap \Sigma_{m,n}\neq \emptyset$, since a general linear section of~$\Sigma_{m,n}^{\mathbb C}$ is smooth by Bertini's theorem, $L\cap \Sigma_{m,n}$ consists of smooth points and hence is Zariski-dense in $L^{\mathbb C}\cap \Sigma_{m,n}^{\mathbb C}$. Therefore, we can find $d+1$ points in  $L\cap \Sigma_{m,n}$ that are general in~$L^{\mathbb C}\cap \Sigma_{m,n}^{\mathbb C}$ and hence span $L$.
\end{proof}

We now let $T\in\mathbb R^{m\times n\times \ell}$ be a Gaussian tensor with slices matrices $T_1,\ldots,T_\ell\in \mathbb R^{m\times n}$. With probability one, their linear span has dimension $\ell$ and hence defines a random point in the Grassmannian $G(\ell, mn)$. We have the following characterization of its distribution.
\begin{lemma}\label{lem_uniform}
Let $T\in\mathbb R^{m\times n\times \ell}$ be a Gaussian tensor with slices $T_1,\ldots,T_\ell\in \mathbb R^{m\times n}$. Let $L:=\langle T_1,\dots, T_\ell\rangle$ be the linear span of the slices; then the random linear subspace $L$ is uniform in the Grassmannian $G(\ell, mn)$.
\end{lemma}  
\begin{proof}
Since Gaussian matrices are invariant under multiplication by orthogonal matrices, the distribution of $L=\langle T_1,\ldots,T_\ell\rangle$ is $O(mn)$-invariant. Since there is a unique such distribution, $L$ must be uniform.
\end{proof}


\subsection{A Probabilistic Version of Friedland's Theorem}\label{sec:friedland_prob}

The following provides a probabilistic version of Friedland's theorem (\cref{thm:Fri2.4}), using the notation we have just introduced.  

\begin{proposition}\label{prop:generalsegre}
Let $T \in \R^{m\times n \times \ell}$ be a Gaussian tensor, with slices $T_1,\dots,T_\ell \in \R^{m\times n}$. Let $L:=\langle T_1,\dots, T_\ell\rangle_{\mathbb P}$ be the projectivization of the linear span of the slices. Then, with probability one, 
$$\mathrm{rank}(T)\leq \ell \quad \Longleftrightarrow\quad L \cap \Sigma_{m,n} \text{ contains at least $\ell$ points.}$$
\end{proposition}
\begin{proof}
Since the rank of $T$ is always bounded from above by $mn$, the proposition is trivial if $\ell \geq mn$, since with probability one $L=\P^{mn-1}$. Thus we can assume $\ell < mn$. 

Assume that $\mathrm{rank}(T) \leq \ell$. Then by \cref{thm:Fri2.4}, there exist real rank-one $m\times n$ matrices $\{x_i \otimes y_i\}_{i=1,\dots,\ell}$ such that their linear span has dimension $\ell$ and contains the linear span of $T_1,\dots,T_\ell$. Since $T$ has Gaussian independent entries, with probability one its slices span an $\ell$-dimensional subspace. This implies 
$L= \langle x_1\otimes y_1,\dots,x_\ell\otimes y_\ell\rangle_{\mathbb P}.$ Since the $\{x_i \otimes y_i\}_{i=1,\dots,\ell}$ are linearly independent, projectively they correspond to different points, proving that $L \cap \Sigma_{m,n}$ contains at least $\ell$ points. 

Vice versa, assume that the intersection $L \cap \Sigma_{m,n}$ contains at least $\ell$ real points. \cref{prop_linear_sections} implies that, with probability one, $L$ is spanned by rank-one matrices. Therefore, 
$\ell$ of the points in $L \cap \Sigma_{m,n}$, say $\{x_i\otimes y_i\}_{i=1,\dots,\ell}$, are linearly independent. Then, we have $L=\langle x_1\otimes y_1,\dots,x_\ell\otimes y_\ell\rangle_{\mathbb P}$. By \cref{thm:Fri2.4} it follows that $\mathrm{rank}(T)\leq \ell$. 
\end{proof}

We prove the first part of \cref{main0_intro} using \cref{thm:Fri2.4} and \cref{prop_linear_sections}. 
\begin{proof}[Proof of Theorem \ref{main0_intro} -- first part]
Let $T \in \R^{m\times n\times \ell}$ be a Gaussian tensor with slices $T_1,\dots,T_{\ell} \in \R^{m\times n}$ and let $L=\langle T_1,\dots,T_{\ell}\rangle$ be the linear span of the slices. By \cref{lem_uniform},~$L$ is uniform in~$G(\ell,mn)$. In particular, with probability one we have that $\dim(L)=\ell$ and so the rank of $T$ is at least $\ell$ by \cref{thm:Fri2.4}. Consequently, typical ranks are bounded from below by $\ell$. 

Next, we prove that typical ranks are bounded from above by $\ell+1$. 
We denote the projectivization of $L\in G(\ell,mn)$ also by $L$. 
If $L\cap \Sigma_{m,n}$ contains at least $\ell$ real points, by \cref{prop:generalsegre} the rank of $T$ is bounded by $\ell$. If this is not the case, let $x \otimes y$ be a general rank-one matrix such that $x \otimes y \notin L$ and define $\tilde L := L + \langle x\otimes y \rangle \in G(\ell+1,mn)$. We will now prove that $\tilde L$ is general.

Let $G_{\C}(d,N)$ denote the complex Grassmannian and recall that $\overline{G(d,N)}=G_{\C}(d,N)$ and $\overline{\Sigma_{m,n}}=\Sigma_{m,n}^{\C}$, where we are taking complex Zariski closures. Consider the rational map 
\begin{align*}
    \phi : G_{\C}(\ell,mn) \times \Sigma^{\C}_{m,n} &\dashrightarrow G_{\C}(\ell+1,mn) ,\quad (L,p)   \to  L + \langle p \rangle \ .
\end{align*}
Denote by $A := \phi(G(\ell,mn)\times \Sigma_{m,n})$. We want to prove that elements in $A$ are general in $G_{\C}(\ell+1,mn)$, that is, $\overline{A}=G_{\C}(\ell+1,mn)$. Note that $G_{\C}(\ell,mn)\times \Sigma_{m,n}^{\C}$ and $G_{\C}(\ell+1,mn)$ are irreducible and $\overline{G(\ell,mn)\times \Sigma_{m,n}}=G_{\C}(\ell,mn)\times \Sigma_{m,n}^{\C}$. We have that the Zariski closure of $A$ is the same of the Zariski closure of the image of $\phi$ (see, e.g., \cite[Lemma 1.2]{multiview}). By the last part of \cref{prop_linear_sections}, a general point in $G_{\C}(\ell+1,mn)$ is spanned by rank-one matrices, which implies
\begin{align*}
    \overline{A}=\overline{\phi(G_{\C}(\ell,mn)\times \Sigma_{m,n}^{\C})}=G_{\C}(\ell+1,mn).
\end{align*}
This proves that $\tilde L$ is general in $G_{\C}(\ell+1,mn)$. Since by construction $\tilde L\cap \Sigma_{m,n}\neq \emptyset$, the last part of \cref{prop_linear_sections} implies that $\tilde L$ is spanned by rank-one matrices. By \cref{thm:Fri2.4} this implies that the rank of $T$ is bounded by $\ell+1$. Thus, with probability one, the rank of $T$ is always bounded by $\ell+1$.
\end{proof}

We now prove  \cref{main1_intro}.
\begin{proof}[Proof of Theorem \ref{main1_intro}]
By \cref{main0_intro}, $\ell$ is the smallest typical rank. Therefore, with probability one, it holds that $\mathrm{rank}(T)\geq \ell$. Together with \cref{prop:generalsegre}, this shows that  $\mathrm{rank}(T) = \ell$ if and only if $L \cap \Sigma_{m,n}$ contains at least $\ell$ points. The equality of probabilities follows using \cref{lem_uniform}. 

Furthermore, \cref{prop_linear_sections} implies that for $\ell > (m-1)(n-1)+1$, with probability one, a single point in $L \cap \Sigma_{m,n}$ suffices to show that $L$ is spanned by rank-one matrices. Hence, $L \cap \Sigma_{m,n}$ contains at least dim$(L) = \ell$ points.
\end{proof}

Finally, we prove the second part of \cref{main0_intro}. 
\begin{proof}[Proof of Theorem \ref{main0_intro} -- second part]
We prove that $\ell$ is always a typical rank.

In the case $\ell>(m-1)(n-1)+1$, the second part of \cref{main1_intro} shows that $\ell$ is a typical rank if and only if ${\Prob} \{L \cap \Sigma_{m,n} \neq \emptyset\}>0$. 
To show that this probability is positive, we can simply take a point $A\in \Sigma_{m,n}$ and a general linear space $K$ that passes through $A$ such that $K$ is not tangential to $\Sigma_{m,n}$ at $A$. Then there exists a neighborhood of $K$ such that for all $L$ in this neighborhood, we have $L \cap \Sigma_{m,n} \neq \emptyset$, and so ${\Prob} \{L \cap \Sigma_{m,n} \neq \emptyset\}>0$. 

The case $\ell=(m-1)(n-1)+1$ is more difficult. By \cref{main1_intro}, we have to prove that ${\Prob} \{\#(L \cap \Sigma_{m,n})\geq \ell \}>0$. To see this, denote by $k:=mn-\ell = m+n-2$ and 
consider general points $a_1,\ldots,a_k\in\mathbb R^m$ and $b_1\ldots,b_k\in\mathbb R^n$. They define a linear subspace $\{M\in\mathbb R^{m\times n} \mid a_1^TMb_1 = \cdots = a_k^TMb_k=0\}$ of dimension $\ell$. For $x\otimes y = xy^T\in \Sigma_{m,n}$ these equations become
\begin{equation}
   \label{special_slice}(a_1^T x)(b_1^Ty) = \cdots = (a_k^T x)(b_k^Ty)=0.
\end{equation}
Since the $a_i$ and $b_i$ are general, at most $m-1$ of the $a_i^T x$ and at most $n-1$ of the $b_j^Ty$ can be simultaneously zero. Consequently, if all of the products in (\ref{special_slice}) are supposed to be zero, we can choose $x$ such that exactly $m-1$ of the $a_i^T x$ are zero and this then determines~$y$. This shows that there are $\mathrm{deg}(\Sigma_{m,n}^{\mathbb C})= \tbinom{m+n-2}{m-1}\geq \ell $ (see~(\ref{degree_segre})) real solutions of (\ref{special_slice}). Since they are all regular solutions, they can be continued in an open neighborhood of the system (\ref{special_slice}). Hence, $\Prob\{\#(L \cap \Sigma_{m,n})\geq \ell\}>0$. 
\end{proof}

\begin{remark}
If in \cref{special_slice} we use pairs of complex conjugate rank one-matrices (and one additional real rank one matrix if $k$ is odd), we get the following number of real solutions:
$$\begin{cases}
0, &\text{ if $m$ and $n$ are even}\\[0.5em]
\tbinom{(m+n-3)/2}{(m-1)/2}, &\text{ if $m$ is odd and $n$ is even}\\[0.5em]
\tbinom{(m+n-3)/2}{(m-2)/2}, &\text{ if $m$ is even and $n$ is odd}\\[0.5em]
\tbinom{(m+n-2)/2}{(m-1)/2}, &\text{ if $m$ and $n$ are odd}
\end{cases}$$
This is the number $\alpha(m,n)$ from \cite{SMS21}. In particular, this shows that if $m$ and $n$ are even, we always have two typical ranks. Notice that $\alpha(3,3)=2$, but we know from \cref{cor_cubic_lines} below that $\Prob\{\# (L\cap \Sigma_{3,3}) = 0\}>0$, so $\alpha(m,n)$ is not a lower bound for the number of points in $L\cap \Sigma_{m,n}$. 
\end{remark}

Another consequence of \cref{prop:generalsegre} is that, if $2 \leq m \leq n \leq \ell < (m-1)(n-1)+1$, then with probability one, the projectivization of the linear span of the $m\times n$ slices of a Gaussian tensor $T \in \R^{m\times n \times \ell}$ will have no intersection with the Segre variety $\Sigma_{m,n}$, since its dimension is lower than the codimension of $\Sigma_{m,n}$. Therefore the probability that the condition given in \cref{prop:generalsegre} to have $\mathrm{rank}(T)\leq\ell$ is satisfied is $0$. This means that typical ranks are all greater than or equal to $\ell +1$. 

Deciding whether or not $\ell+1$ is a typical rank is more subtle. In view of \cref{thm:Fri2.4}, we can see that having rank $\ell+1$ is equivalent to the existence of an $(\ell+1)$-plane containing the linear span of the slices of $T$, such that its projectivization intersects $\Sigma_{m,n}$ in at least $\ell+1$ linearly independent real points. The subtlety is that we are not asking for a general $(\ell+1)$-plane containing the span of the slices to have this property, but we only require the existence of a single such plane. Therefore, dimension counting arguments, which only provide general results, cannot be used.

\subsection{Typical Ranks of Tall Tensors}\label{sec:tall_tensors}
We conclude this section by giving another proof of a result of ten Berge. 
In \cite{tenBerge2000} for $2\leq m \leq n \leq \ell$, an $m\times n\times \ell$ tensor is defined to be \emph{tall} if $(m-1)n < \ell < mn$, and it is shown that the only typical rank for such tensors is always $\ell$. We now will use \cref{prop:generalsegre} once again to give a different proof of this result. 

Let $T \in \R^{m\times n \times \ell}$ be a tall Gaussian tensor and denote by $T_1,\dots,T_{\ell}$ its $m\times n$ slices. Let $k:=mn-\ell$. With probability one, the dimension of $\langle T_1,\dots,T_{\ell}\rangle$ is $\ell$ and therefore there exist $A_1,\dots,A_k \in \R^{m\times n}$ such that 
$$\langle T_1,\dots,T_{\ell}\rangle = \{M\in\P^{mn-1} \mid \mathrm{Trace}(M^TA_1) = \cdots = \mathrm{Trace}(M^TA_k)=0\}$$
(see \eqref{def_IP}). For a rank-one matrix $x\otimes y$ with $x \in \R^m$ and $y\in \R^n$, the condition $x\otimes y \in \langle T_1,\dots,T_{\ell}\rangle$ is equivalent to the system of equations
\begin{align*}
x^TA_1y= \cdots = x^TA_ky = 0, 
\end{align*}
which is in turn equivalent to 
\begin{align}\label{leftkernelA}
    x^T \cdot \mleft(A_1 y \dots A_k y\mright)=0,
\end{align}
i.e. $x^T$ is in the left kernel of the $m \times k$ matrix having columns $A_iy$ for $i=1,\dots,k$. For every $i=1,\dots,n$ define $B_i \in \R^{m\times k}$ as the matrix containing the $i$-th columns of the matrices $A_1,\dots,A_k$, i.e. the $j$-th column of $B_i$ is the $i$-th column of $A_j$. Then we have
\begin{align*}
    (A_1y,\dots,A_ky) = (B_1y_1 + \dots + B_ny_n)
\end{align*}
where $y=(y_1,\dots,y_n)$. Equation \eqref{leftkernelA} becomes 
\begin{align}\label{leftkernelB}
    x^T \cdot (B_1y_1 + \dots + B_ny_n)=0.
\end{align}
Since $T$ is tall, we have $k<m$ and equation \eqref{leftkernelB} always has non-trivial solutions. This means that for every $y \in \R^n$ we can find $x \in \R^m$ such that $x\otimes y$ belongs to $\langle T_1,\dots,T_{\ell}\rangle$. By \cref{prop:generalsegre} it follows that $\mathrm{rank}(T)\leq \ell$. By \cref{main0_intro}, $\ell$ is also the smallest typical rank, which implies that $\ell$ is the only typical rank.

\bigskip
\section{Random \texorpdfstring{$3\times 3\times 5$}{TEXT} tensors and cubic surfaces}
\label{sec:3x3x5}

In this section, we focus on the special case of Gaussian tensors of the format $3\times 3\times 5$. 

ten Berge \cite{tenBerge2004} showed that tensors in $\R^{3\times 3\times 5}$ admit two typical ranks: $5$ and~$6$. In this section, we express the probability that a Gaussian tensor $T$ has rank $5$ as the probability that the $27$ lines on a certain random cubic surface are all real. 

Recall that blowing up six points in $\mathbb C\mathrm P^2$ that are in general position (no three points are collinear and not all lie on a conic) produces a smooth cubic surface, and indeed every smooth cubic surface can be generated in this way \cite{Clebsch1866}.
\cref{lem_uniform}  implies that a $3\times 3\times 5$ Gaussian tensor $T$ induces a uniform subspace $L\subset \P^8$ of dimension~$4$. By \cref{main1_intro}, with probability one, if $T$ has rank $5$, it intersects $\Sigma_{3,3}$ in 6 points $x_1\otimes y_1,\ldots,x_6\otimes y_6$. We can therefore associate a random cubic surface to~$T$ that is the blow up of $\P^2$ in $x_1,\dots,x_6$ or $y_1,\dots,y_6$.

 Our results are based on the following theorem. Note that $\mathrm{codim}\, \Sigma_{3,3}^{\mathbb C} = 4$ and $\mathrm{deg}\, \Sigma_{3,3}^{\mathbb C} = 6$, so a general $4$-dimensional subspace of $\mathbb C\mathrm P^8$ intersects $\Sigma_{3,3}^{\mathbb C}$ in $6$ points.
\begin{theorem}
Let $L$ be a four-dimensional general linear subspace of $\P^{8}$. Then $L$ intersects $\Sigma_{3,3}^{\mathbb C}$ in the $6$ points  $x_1\otimes y_1,\dots,x_6 \otimes y_6 \in \mathbb C\mathrm P^{8}$ if and only if there exists a cubic surface $S$ that results both from blowing up $\mathbb C\mathrm P^2$ in $x_1,\dots,x_6$ as well as from blowing up~$\mathbb C\mathrm P^2$ in $y_1,\dots,y_6$, such that
$$
        S = \{[z_0:\dots:z_3] \in \P^3 \mid \det(z_0 M_0 + \dots + z_3 M_3)=0\},
$$
where $\langle M_0, \dots,M_3 \rangle = L^\perp$ is the orthogonal complement (see \eqref{def_OC}).
\label{thm:CAET}
\end{theorem}
\begin{proof}
The theorem is proved in \cite[Theorem 6.9]{CAET2023} for a general linear subspace  of $\mathbb C\mathrm P^{8}$. Since the real Grassmannian is Zariski dense in the complex Grassmannian, it also holds for a general subspace in $\P^8$.
\end{proof}
\begin{remark}
\cite[Theorem 6.9]{CAET2023} also proves that the exceptional lines $\{h_1,\dots,h_6\}$ from blowing up in $\{x_1,\dots,x_6\}$ and the exceptional lines $\{h_1',\dots,h_6'\}$ from blowing up in $\{y_1,\dots,y_6\}$ form a Schl\"afli double six on $S$.
\end{remark}
\begin{remark}
Remarkably, \cite{CAET2023} is a paper that studies a problem in computer vision. The connection to ranks of random tensors is surprising.
\end{remark}
\begin{example}
Consider the real rank-one matrices $x_i\otimes y_i$ for $i=1,\ldots,6$ where
$$
\begin{array}{cccccc}
    x_1 = (1,0,0)
    &x_2 = (0,1,0)  
    &x_3 = (0,0,1)  
    &x_4 = (1,1,1)  
    &x_5 = (3,5,1) \\
    y_1 = (1,0,0)
    &y_2 = (0,1,0)
    &y_3 = (0,0,1)
    &y_4 = (1,1,1)
    &y_5 = (8,2,1) 

\end{array}
$$
and $x_6 = (-\tfrac{1}{3}, \tfrac{7}{5}, \tfrac{3}{17})$ and $y_6 = (-\tfrac{4}{3}, \tfrac{2}{5}, -\tfrac{1}{17})$. The first five rank-one matrices are linearly independent, but the sixth is linearly dependent on the first five. Therefore, the six rank-one matrices span a subspace $L \in G(5,9)$. The orthogonal complement is
$$L^\perp = \left\{ \left[ \begin{array}{ccc}
    0 & 3z_0 & 3z_2 \\
    -37z_0-2z_1-5z_2+z_3 & 0 & 3z_3 \\
    34z_0-z_1+2z_2-4z_3 & 3z_1 & 0
\end{array} \right] \, \biggm| \,z_0,\ldots,z_3\in\mathbb R \right\}.$$
The cubic surface $S$ constructed in \cref{thm:CAET} is thus given by
$$
    S = \{z\in \P^3 \mid z_0 z_3 (34z_0-z_1+2z_2-4z_3) + z_1 z_2 (-37z_0-2z_1-5z_2+z_3) = 0\} .
    \label{eq:S}
$$
It contains $27$ real lines.
\end{example}

\begin{figure}
\begin{center}
\includegraphics[height = 4cm]{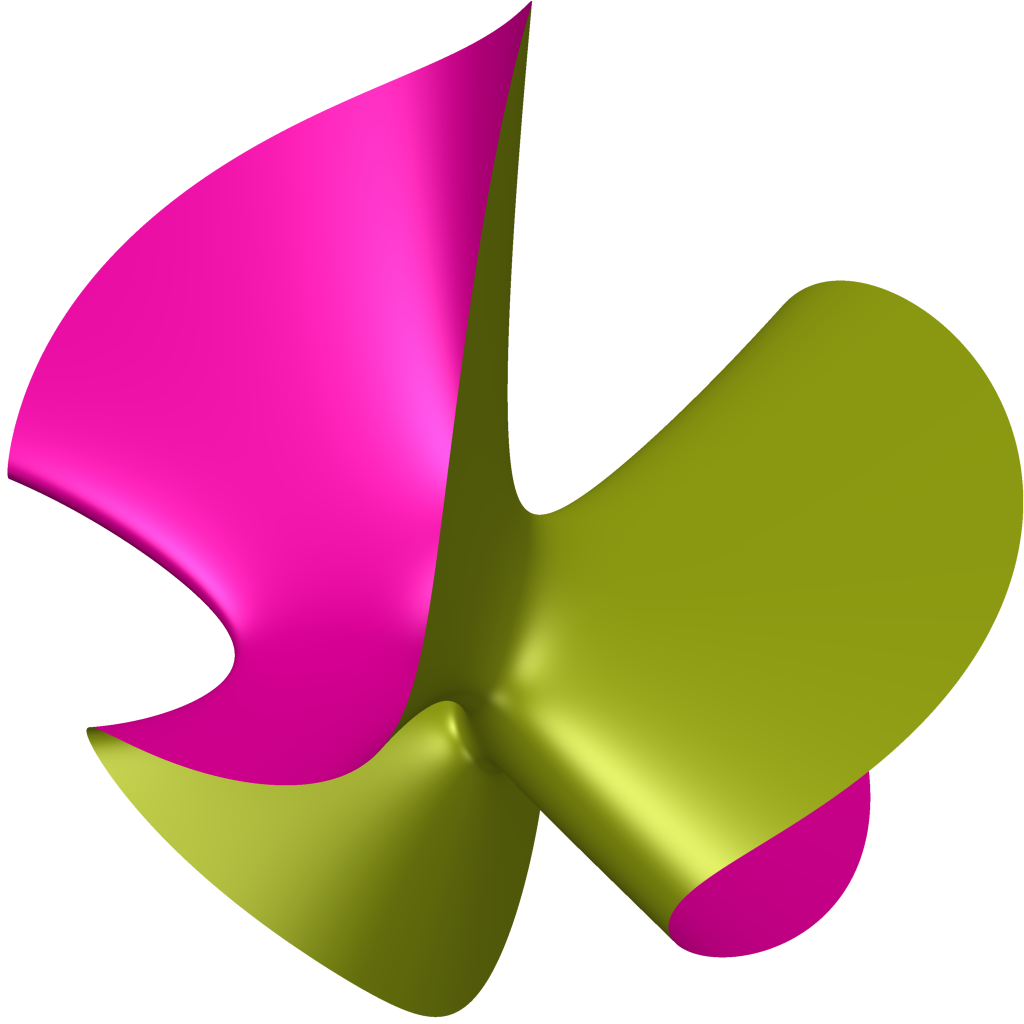}\qquad
\includegraphics[height = 4cm]{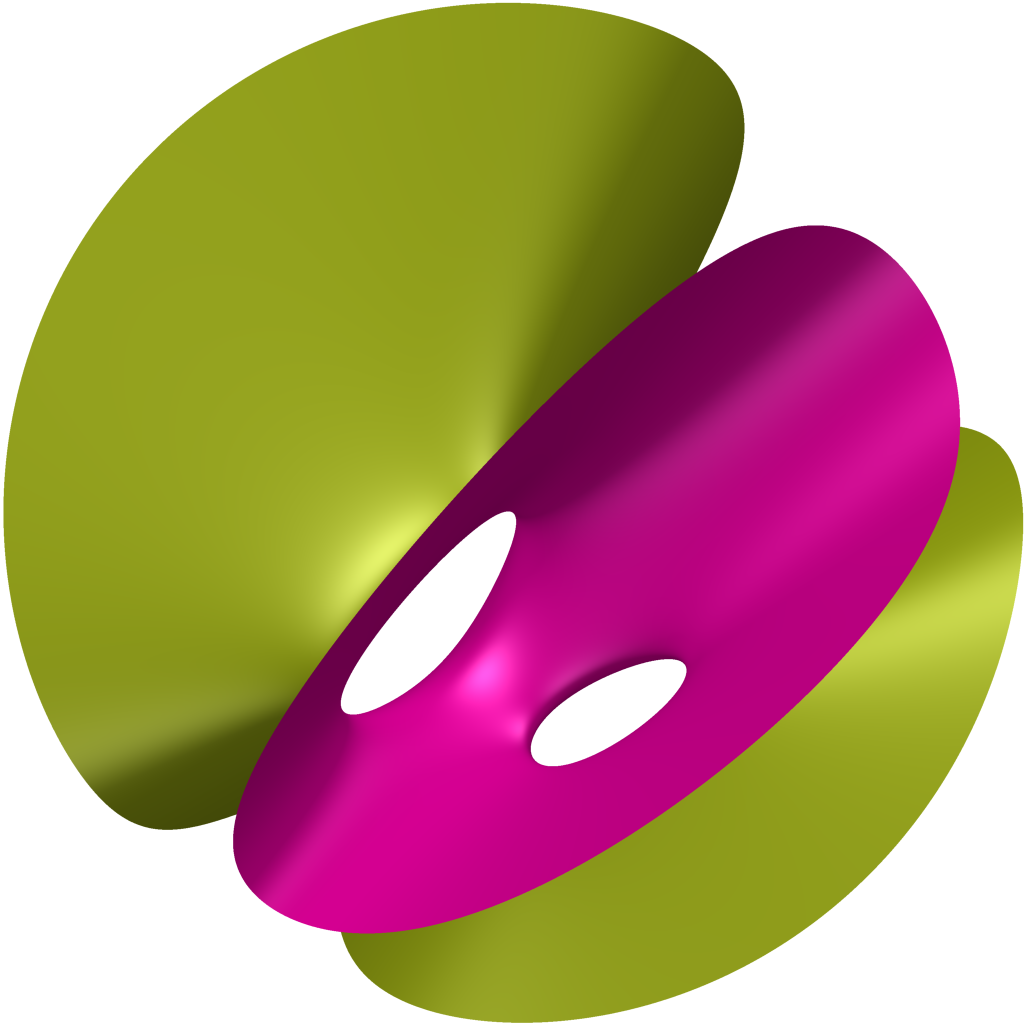}
\end{center}
\caption{The cubic surface $S$ in \cref{eq:S} intersected with the hyperplane $z_0=1$ seen from two different viewpoints. The pictures were created using  \href{https://www.imaginary.org/de/program/surfer}{\texttt{Surfer}}. 
}
\end{figure}

In 1849, Cayley and Salmon proved that every smooth cubic surface contains 27 lines \cite{Cayley1849}. Since their classical proof, many different proofs have been produced throughout the years. It was shown in 1858 by \cite{Schlaefli} that a real cubic surface can only have $3$, $7$, $15$, or $27$ real lines, leading to a classification of cubic surfaces. We will need  the following result due to Polo-Blanco and Top.
\begin{proposition}\cite[Proposition 2.1]{PBT2008}\label{prop_PBT}
Let $S$ be a smooth, real cubic surface.  If $S$ is obtained as the blow-up of
\begin{enumerate}
    \item $6$ real points in the plane, it contains $27$ real lines;
    \item $4$ real points and a pair of complex conjugate points in the plane, it contains $15$ real lines;
    \item $2$ real points and $2$ pairs of complex conjugate points in the plane, it contains $7$ real lines
    \item $3$ pairs of complex conjugate points in the plane, it contains $3$ real lines.
\end{enumerate}
\end{proposition}
\begin{remark}
There is a fifth case in \cite[Proposition 2.1]{PBT2008}, but this case appears with probability zero, so we do not include it. 
\end{remark}
\begin{corollary}\label{cor_cubic_lines}
 Let $M_0,M_1,M_2,M_3\in\mathbb R^{3\times 3}$ be independent Gaussian matrices and consider the random cubic surface  
$
        S = \{[z\in \P^3 \mid \det(z_0 M_0 + \dots + z_3 M_3)=0\}.
$
Let $L\in G(5,9)$ be a uniform subspace and denote also by $L$ its projectivization.
Denote by
$$p_i := \Prob\{\# (L\cap \Sigma_{3,3}) = i\} \quad\text{and}\quad q_j:= \Prob\{S \text{ contains $j$ real lines}\}.$$
Then, 
$$p_0 = q_3,\quad 
p_2 = q_7,\quad 
p_4 = q_{15},\quad 
p_6 = q_{27}.$$
\end{corollary}
\begin{proof}
Since $\mathrm{deg}\, \Sigma_{3,3}^{\mathbb C} = 6$, a general real projective linear subspace $L$ intersects $\Sigma_{3,3}$ in either $0,2,4$ or $6$ real points. Therefore, $p_0+p_2+p_4+p_6 = 1$. On the other hand, a smooth cubic surface contains either $3,7,15$ or $27$ real lines. Since $S$ is smooth with probability one,~$q_3+q_7+q_{15}+q_{27}=1$.  

By \cref{lem_OC}, $L\in G(5,9)$ is uniform if and only if $L^\perp\in G(4,9)$ is uniform. Moreover, by \cref{lem_uniform} $\langle M_0,\ldots,M_3\rangle \in G(4,9)$ is also uniformly distributed. It follows from \cref{thm:CAET} that 
$p_i$ is the probability that $S$ is obtained as the blow up of $i$ real points and $(6-i)/2$ pairs of complex conjugate points in the plane.  
\cref{prop_PBT} implies that $p_6 \leq q_{27}$, $p_4 \leq q_{15}$, $p_2 \leq q_{7}$  and~$p_0 \leq q_{3}$. If one of those was a strict inequality, then $q_3+q_7+q_{15}+q_{27} < p_0+p_2+p_4+p_6 = 1$, which is a contradiction. Hence, all are equalities.
\end{proof}

The stage is now set to relate the probability that a random real $3\times 3\times 5$ tensor has rank 5 to the probability that all $27$ lines on a random cubic surface are real.
\begin{proof}[Proof of Theorem \ref{thm:rank-reallines_intro}]
As before, we write $p_6 = \Prob\{\# (L\cap \Sigma_{3,3}) = 6\}$. By \cref{main1_intro}, $\Prob\{\text{rank }T = 5\}$ is the probability that for a uniform subspace $L\in G(5,9)$, its projectivization $L$ intersects $\Sigma_{3,3}$ in at least $5$ points. Since $L$ intersects $\Sigma_{3,3}$ in either $0$, $2$, $4$ or $6$ points with probability one, we have 
$$\Prob\{\text{rank }T = 5\} = p_6.$$
The statement follows from \cref{cor_cubic_lines}.
\end{proof}

We conclude this section by studying the expected number of real lines on our random cubic. Denote by
$$E: = \mathbb E\, \text{ number of real lines on $S$},$$
where $S = \{\det(z_0 M_0 + \dots + z_3 M_3)=0\}$ with $M_i$ independent and Gaussian. We show in \cref{ex1} below that the expected number of points in $L\cap \Sigma_{3,3}$ for uniform $L$ is $3$. Together with \cref{cor_cubic_lines} we obtain the following system of equalities and inequalities:
$$
\begin{cases}
0\leq p_0,p_2,p_4,p_6\leq 1\\
p_0+p_2+p_4+p_6 = 1\\
2p_2 + 4p_4 + 6p_6 = 3\\
3p_0 + 7p_2 + 15p_4 + 27p_6 = E    
\end{cases}$$
Treating $p_0,p_2,p_4,p_6$ and also $E$ as variables, the solution set of this system is a two-dimensional polytope. Projected to the $(E,p_6)$-plane this polytope is given as the convex hull of the four points $(11,0),(12,0),(12,\tfrac{1}{4}),(15,\tfrac{1}{2})$ 
From this we obtain the following nontrivial bound for $E$.
\begin{corollary}\label{cor_E}
 We have $11\leq E\leq 15$.
\end{corollary}
The expected number of lines for other models of random cubic surfaces has been computed in \cite{ABM2021, ACL2010, BLLP2019}. Comparing \cref{cor_E} to these results,~$E$ seems to be comparably large.

\begin{figure}
\begin{center}
\begin{tikzpicture}
\draw[dashed, ->] (8,0) -- (16,0);
\draw[dashed, ->] (8,0) -- (8,2.5);
\draw (8,2.5) node[left]{$p_6$};
\draw (16,0) node[below]{$E$};

\fill[teal, opacity = 0.5] (11,0) -- (12,0) -- (15,2) -- (12,1) -- cycle;
\draw (11,0) node{$\bullet$};
\draw (11,0) node[below left]{$(11,0)$};
\draw (12,0) node{$\bullet$};
\draw (12,0) node[below right]{$(12,0)$};
\draw (15,2) node{$\bullet$};
\draw (15,2) node[above right]{$(15,\tfrac{1}{2})$};
\draw (12,1) node{$\bullet$};
\draw (12,1) node[above left]{$(12,\tfrac{1}{4})$};

\end{tikzpicture}
\end{center}
\caption{\label{fig_polygon}The polygon of possible values for $(E,p_6)$, where $E$ is the expected number of real lines on the random cubic, and $p_6$ the probability that it has 27 real lines.}
\end{figure}
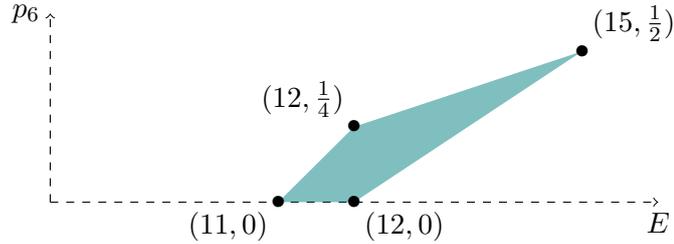

\bigskip
\section{Asymptotics and heuristics}
\label{sec:heur}
In \cref{main1_intro} we proved that the rank of a random tensor $T \in \R^{m\times n \times \ell}$ with $\ell:=(m-1)(n-1)+1$ depends on whether the corresponding random linear subspace spanned by its slices projectively intersects the Segre variety in enough real points. In this section, we will compute the average number of real points in such an intersection and compute its asymptotic for $m$ fixed as $n \rightarrow +\infty$. Comparing the growth of this expectation with the growth of the required number of real points, we obtain some insight into how the probability of having rank $\ell$ behaves. \\

We start by stating the real projective version of Howard's kinematic formula \cite{Howard}. Recall that given a Riemannian manifold $(M,g)$, a submanifold $N \overset{\iota}{\hookrightarrow} M$ inherits an induced Riemannian metric by pullback through $\iota$, denoted as $\iota^*g$. By $\mathrm{vol}(N)$ we mean the total volume of $N$ with respect to the induced metric (if $\mathrm{dim}(N)<\mathrm{dim}(M)$ its volume with respect to $g$ is $0$). If $N$ is compact, $\mathrm{vol}(N)$ will always be finite. In the case $M=\P^{mn-1}$ with the Fubini-Study metric and $N=\Sigma_{m,n}$, Howard's kinematic formula yields 
$$\mathbb E\, \#(L\cap \Sigma_{m,n}) = \frac{\mathrm{vol}(\Sigma_{m,n})}{\mathrm{vol}(\P^{m+n-2})},$$
where $L$ corresponds to a uniform subspace in $G(\ell,mn)$. 

Moreover, $\mathrm{vol}(\Sigma_{m,n})=\mathrm{vol}(\P^{m-1})\cdot \mathrm{vol}(\P^{n-1})$ and $\mathrm{vol}(\P^{k-1})=\frac{\pi^{\frac{k}{2}}}{\Gamma(\frac{k}{2})}$, leading to 
\begin{align}\label{expectation}
    \mathbb{E}(\#\{\Sigma_{m,n}\cap L\})=\sqrt{\pi}\cdot\frac{\Gamma(\frac{m+n-1}{2})}{\Gamma(\frac{m}{2})\Gamma(\frac{n}{2})}.
\end{align}
\begin{example}\label{ex1}
    For $m=3$, Equation \eqref{expectation} becomes
$$
        \mathbb{E}(\#\{\Sigma_{3,n}\cap L\})= n.
$$
    Therefore, for a random tensor $T \in \R^{3\times n \times (2n-1)}$ with the usual notation, \cref{main1_intro} combined with Markov's inequality yields
    \begin{align*}
        \Prob\{\mathrm{rank}(T)=2n-1\}=\Prob\{\#\{\Sigma_{3,n}\cap L\}\geq 2n-1\}\leq \frac{n}{2n-1} < 1. 
    \end{align*}
    This gives another proof of the fact that there are two typical ranks when $m=3$ and $\ell = (m-1)(n-1)+1$, since it implies that $\Prob\{\mathrm{rank}(T)>2n-1\}$ is positive. This is also an alternative proof to \cite[Theorem 5.5]{SMS21} for the case $m=3$. 
\end{example}

We are now interested in the asymptotics of \eqref{expectation} when $m$ is fixed and $n\to +\infty$. In order to compute this, we consider two cases, depending on the parity of $m$. 
\begin{proposition}
\label{lemma:expectation}
Let $L$ be a uniform projective $(m-1)(n-1)$-dimensional subspace.
 For $m$ fixed, as $n\to +\infty$, we have the asymptotic expression
    \begin{align*}
        \mathbb{E}(\#\{\Sigma_{m,n}\cap L\})=c\cdot n^{\frac{m-1}{2}}\mleft(1+\mathcal{O}\mleft(\frac{1}{n}\mright)\mright),
    \end{align*}
    where $c = \frac{1}{(m-2)!!}$ for $m$ odd and $c=\sqrt{\frac{\pi}{2}}\frac{1}{(m-2)!!}$ for $m$ even. 
\end{proposition}
\begin{proof}
We first consider the case when $m$ is odd.
Let  $m=2k+1$ for some $k\geq 1$. In this case, \eqref{expectation} reads
    \begin{align*}
        \mathbb{E}(\#\{\Sigma_{2k+1,n}\cap L\})=\sqrt{\pi}\cdot \frac{\Gamma(\frac{n}{2}+k)}{\Gamma(k+\frac{1}{2})\Gamma(\frac{n}{2})}.
    \end{align*}
    By the multiplicative property of the Gamma function $\Gamma(x+1)=x\Gamma(x)$, we obtain 
    \begin{align*}
        \Gamma\left(\frac{n}{2}+k\right)&=\Gamma\left({\frac{n}{2}}\right)\prod_{i=0}^{k-1}\left(i+\frac{n}{2}\right), \quad \text{and}\quad 
        \Gamma\left(\frac{1}{2}+k\right)=\sqrt{\pi}\prod_{i=0}^{k-1}\left(i+\frac{1}{2}\right),
    \end{align*}
 leading to
    \begin{align}\label{oddexpectation}
        \mathbb{E}(\#\{\Sigma_{2k+1,n}\cap L\})=\prod_{i=0}^{k-1}\frac{\left(\frac{n}{2}+i\right)}{\left(\frac{1}{2}+i\right)}.
    \end{align}
    This is a polynomial of degree $k=\frac{m-1}{2}$ in $n$, with leading coefficient given by
    \begin{align*}
        \frac{(\frac{1}{2})^k}{\prod_{i=0}^{k-1}\frac{2i+1}{2}}=\frac{1}{\prod_{i=0}^{k-1}(1+2i)}=\frac{1}{(2k-1)!!},
    \end{align*}
    where $(2k-1)!!=(2k-1)\cdot (2k-3)\cdot \dots \cdot 3\cdot 1$ is the double factorial. Thus we obtain the asymptotic expression 
        \begin{align}\label{oddasymptotic}
            \mathbb{E}(\#\{\Sigma_{m,n}\cap L\}) = \frac{1}{(m-2)!!}\, n^{\frac{m-1}{2}}\mleft(1+\mathcal{O}\mleft(\frac{1}{n}\mright)\mright) \quad \text{for} \ m \ \text{odd}.
        \end{align}

We now consider the case when $m$ is even.
Let $m=2k$ for some $k\geq 2$. In this case, \eqref{expectation} reads
    \begin{align*}
        \mathbb{E}(\#\{\Sigma_{2k,n}\cap L\})=\sqrt{\pi}\frac{\Gamma\left(\frac{n-1}{2}+k\right)}{\Gamma(k)\Gamma\left(\frac{n}{2}\right)}.
    \end{align*}
    Again by the multiplicative property of the Gamma function we have
    \begin{align*}
        \Gamma\left(\frac{n-1}{2}+k\right)=\Gamma\left(\frac{n+1}{2}\right)\prod_{i=1}^{k-1}\left(\frac{n-1}{2}+i\right).
    \end{align*}
    Using above identity and $\Gamma(k)=(k-1)!$, we obtain 
    \begin{align*}
        \mathbb{E}(\#\{\Sigma_{2k,n}\cap L\})&=\pi \cdot \frac{\prod_{i=1}^{k-1}\mleft(\frac{n-1}{2}+i\mright)}{(k-1)!} \cdot \frac{\Gamma\mleft(\frac{n+1}{2}\mright)}{\Gamma\mleft(\frac{n}{2}\mright)\Gamma\mleft(\frac{1}{2}\mright)} 
        =\pi \cdot \frac{\prod_{i=1}^{k-1}\mleft(\frac{n-1}{2}+i\mright)}{(k-1)!} \cdot \frac{1}{B\mleft(\frac{1}{2},\frac{n}{2}\mright)},
    \end{align*}
    where $B(x,y)=\frac{\Gamma(x)\Gamma(y)}{\Gamma(x+y)}$ is the Beta function. Using the asymptotic $B\left(\frac{1}{2},\frac{n}{2}\right)\sim \sqrt{\frac{2\pi}{n}}$ (see \cite[Equation 43:9:3]{atlas}) and the identity $2^{k-1}(k-1)!=(2k-2)!!=(m-2)!!$, we have
    \begin{align}\label{evenasymptotic}
        \mathbb{E}(\#\{\Sigma_{m,n}\cap L\})=\sqrt{\frac{\pi}{2}}\frac{1}{(m-2)!!}n^{\frac{m-1}{2}}\mleft(1+\mathcal{O}\mleft(\frac{1}{n}\mright)\mright) \quad \text{for} \ m \ \text{even}.
    \end{align}
    \end{proof}

We can now compare the results in \cref{lemma:expectation} with \cref{main1_intro}. When we have $\ell=(m-1)(n-1)+1$, \cref{main1_intro} states that for a random $m\times n \times \ell$ tensor having rank $\ell$ is, with probability one, equivalent to $\#\{\Sigma_{m,n}\cap L\}\geq \ell$. While $\ell$ grows linearly in $n$ for $m$ fixed, the expectation in \eqref{expectation} grows as $n^{\frac{m-1}{2}}$. 

For $m=3$, both quantities grow linearly, but while $\ell$ is given by $2n-1$, the expectation is $n$. Qualitatively, this means that the number of linearly independent points required to be real in order to have rank $\ell=2n-1$ grows faster, albeit with the same order, compared to the expected value of such points. This suggests that the probability that the rank is $2n-1$ should go to $0$ as $n\to +\infty$. 

For $m\geq 4$, while $\ell$ grows linearly, the expected value of real points grows as a higher power of $n$. Qualitatively, this tells us that as $n$ increases, the number of points we require to be real becomes smaller and smaller relative to the expected value. Based on this, we expect that the probability of having rank $\ell$ approaches $1$ as $n \to +\infty$.

\bibliographystyle{alpha}
\bibliography{math}

\end{document}